\documentclass[11pt,reqno,tbtags]{amsart}
\usepackage{amssymb}
\usepackage{ifdraft}
\usepackage{url}
\usepackage[square,numbers]{natbib}

\numberwithin{equation}{section}

\title
{Protected nodes and fringe subtrees in some random trees}

\date{2 October, 2013}

\newcommand\urladdrx[1]{{\urladdr{\def~{{\tiny$\sim$}}#1}}}

\author{Luc Devroye}
\address{School of Computer Science, McGill University, 3480 University Street, 
Montr\'eal, Qu\'ebec, H3A 2A7, Canada}
\email{luc@cs.mcgill.ca}
\urladdrx{http://luc.devroye.org}

\author{Svante Janson}
\thanks{SJ partly supported by the Knut and Alice Wallenberg Foundation}
\address{Department of Mathematics, Uppsala University, PO Box 480,
SE-751~06 Uppsala, Sweden}
\email{svante.janson@math.uu.se}
\urladdrx{http://www2.math.uu.se/~svante/}

\renewcommand\le{\leqslant}
\renewcommand\ge{\geqslant}

\allowdisplaybreaks

\newtheorem{theorem}{Theorem}[section]

\theoremstyle{definition}
\newtheorem{example}[theorem]{Example}

\newtheorem{remark}[theorem]{Remark}

\newtheorem*{ack}{Acknowledgement}

\theoremstyle{remark}

\newenvironment{romenumerate}[1][0pt]{
\addtolength{\leftmargini}{#1}\begin{enumerate}
 }{\end{enumerate}}

\newcounter{oldenumi}
{\setcounter{oldenumi}{\value{enumi}}
\begin{romenumerate} \setcounter{enumi}{\value{oldenumi}}}
{\end{romenumerate}}

\newcounter{thmenumerate}

\newcounter{romxenumerate}   

\newcounter{xenumerate}   

\newcommand{\refT}[1]{Theorem~\ref{#1}}

\newcommand{\refS}[1]{Section~\ref{#1}}

\newcommand{\refE}[1]{Example~\ref{#1}}

\newcommand\REM[1]{{\raggedright\texttt{[#1]}\par\marginal{XXX}}}

\begingroup
  \divide\count255 by 60
  \multiply\count255 by -60
  \advance\count255 by \time
  \ifnum \count255 < 10 \xdef\klockan{\the\count1.0\the\count255}
  \else\xdef\klockan{\the\count1.\the\count255}\fi
\endgroup

\newcommand{\sumk}{\sum_{k=0}^\infty}
\newcommand{\sumki}{\sum_{k=1}^\infty}

\newcommand\set[1]{\ensuremath{\{#1\}}}

\newcommand\xpar[1]{(#1)}
\newcommand\bigpar[1]{\bigl(#1\bigr)}
\newcommand\Bigpar[1]{\Bigl(#1\Bigr)}

\newcommand\lrpar[1]{\left(#1\right)}

\def\rompar(#1){\textup(#1\textup)}    
\newcommand\xfrac[2]{#1/#2}

\newcommand\parfrac[2]{\lrpar{\frac{#1}{#2}}}

\newcommand\Bigparfrac[2]{\Bigpar{\frac{#1}{#2}}}

\def\xexp(#1){e^{#1}}

\newcommand\ntoo{\ensuremath{{n\to\infty}}}

\newcommand\punkt[1]{\if.#1\else.\spacefactor1000\fi{#1}}
\newcommand\iid{i.i.d\punkt}    
\newcommand\ie{i.e\punkt}
\newcommand\eg{e.g\punkt}

\newcommand{\as}{a.s\punkt}

\newcommand{\tend}{\longrightarrow}
\newcommand\dto{\overset{\mathrm{d}}{\tend}}
\newcommand\pto{\overset{\mathrm{p}}{\tend}}

\newcommand\eqd{\overset{\mathrm{d}}{=}}

\newcommand\bbZgeo{\mathbb Z_{\ge0}}

\newcounter{CC}
\newcounter{cc}

\newcommand\E{\operatorname{\mathbb E{}}}
\renewcommand\P{\operatorname{\mathbb P{}}}

\newcommand\Exp{\operatorname{Exp}}
\newcommand\Po{\operatorname{Po}}
\newcommand\Bi{\operatorname{Bi}}

\newcommand\Ge{\operatorname{Ge}}

\newcommand\eps{\varepsilon}

\renewcommand\phi{\xxx}  

\newcommand\cE{\mathcal E}

\newcommand\cL{{\mathcal L}}

\newcommand\cT{{\mathcal T}}

\newcommand\qw{^{-1}}

\newcommand\intoi{\int_0^1}
\newcommand\intot{\int_0^t}
\newcommand\intoo{\int_0^\infty}

\newcommand\dd{\,\mathrm{d}}

\newcommand\dx{d^+}
\newcommand\Tx{T_*}
\newcommand\st{\mathfrak T}
\newcommand\stn{\st_n}
\newcommand\xp{{\mathsf p}}
\newcommand\np{n_\xp}
\newcommand\pp{p_\xp}
\newcommand\ppl{p_{\xp,\ell}}
\newcommand\cep{\cE_\xp}
\newcommand\px{p_*}
\newcommand\pxl{p_{*,\ell}}
\newcommand\pxli{p_{*,\ell-1}}
\newcommand\pxx[1]{p_{*,#1}}
\newcommand\Phit{\Phi_\tau}

\newcommand\sgrt{simply generated random tree}
\newcommand\ws{weight sequence}

\newcommand\wwk{(w_k)_{k=0}^\infty}

\newcommand\ppik{\ensuremath{(\pi_k)_{k=0}^\infty}}
\newcommand\GW{Galton--Watson}
\newcommand\GWt{\GW{} tree}
\newcommand\cGWt{conditioned \GW{} tree}

\newcommand\bbNo{\bbZgeo}
\newcommand\ctnx{\cT_{n,*}}

\newcommand\lprot{$\ell$-protected}
\newcommand\ctt{\tilde T}
\newcommand\fT{\hat T}
\newcommand\ctn{\cT_n}

\begin{document}


\begin{abstract} 
We study protected nodes in various classes of random
rooted trees
by putting them in the general context of fringe subtrees
introduced by Aldous (1991). 
Several types of random trees are considered:
simply generated trees (or conditioned Galton--Watson trees),
which includes several cases treated separately by other authors,
binary search trees and random recursive trees.
This gives unified and simple proofs of several earlier results, as well as
new results.
\end{abstract}

\maketitle

\section{Introduction}\label{S:intro}

Several recent papers study protected nodes in various classes of random
rooted trees, where a node is said to be 
\emph{protected} if it is not a leaf and, furthermore, none of its children
is a leaf.
(Equivalently, a node is protected if and only if the distance to any
descendant that is a leaf is at least 2; for generalizations, see
\refS{Sell}.) 
See \citet{CheonShapiro} (uniformly random ordered trees, Motzkin trees,
full binary trees,  binary trees, full ternary trees),
\citet{Mansour} ($k$-ary trees),
\citet{DuProdinger} (digital search trees),
\citet{MahmoudWard-bst} (binary search trees),
\citet{MahmoudWard-rrt} (random recursive trees),
\citet{Bona} (binary search trees).

The purpose of the present paper is to extend and sharpen some of these
results by putting them in the general context of \emph{fringe subtrees}
introduced by \citet{AldousFringe}.

If $T$ is any rooted tree, and $v$ is a node in $T$, let $T_v$ be the subtree
rooted at $v$. By taking $v$ uniformly at random from the nodes of $T$, we
obtain a random rooted tree which we call the \emph{random fringe subtree} of
$T$ and denote by $\Tx$.

Note that a node $v$ is protected if and only if the subtree $T_v$ has a
protected root. Hence, if $\cep$ is the set of trees that have a protected
root,
then $v$ is protected in $T$ if and only if $T_v\in\cep$.
In particular, taking $v$ uniformly at random, for any given tree $T$,
\begin{equation}\label{ppt}
 \pp(T):=\P(\text{a uniformly random node $v$ is protected})=\P(\Tx\in\cep).
\end{equation}
and we immediately
obtain results for protected nodes from more general results for fringe
subtrees, see \refS{Spf}.

When $T$ is a random tree, we can think of $\Tx$ in two ways, called
\emph{annealed} and \emph{quenched} using terminology from statistical
physics.
In the annealed version we take a random tree $T$ and a uniformly random
node $v$ in it, yielding a random fringe subtree $\Tx$.

In the quenched version we do the random choices in two steps. First we
choose a random tree $T$. We then fix $T$ and choose $v\in T$ uniformly at
random, yielding a random fringe subtree $\Tx$ depending on $T$. We thus
obtain for every choice of $T$ a  probability distribution $\cL(\Tx)$ 
on the set $\st$ of all rooted trees; 
this distribution depends on the random tree $T$ and is
thus a random probability distribution. 
In other words, we consider the conditional distribution $\cL(\Tx\mid T)$ of
$\Tx$ given $T$.
We can now study properties of this random probability distribution. 
Averaging over $T$, we obtain the
distribution of $\Tx$ in the annealed version, so results in the quenched
version are generally stronger than in the annealed version.

Returning to protected nodes, we see that in the quenched point of view, we
consider $\np(T)$, the number of protected nodes in a tree $T$, and
$\pp(T)=\np(T)/|T|$, the probability that a randomly chosen node in $T$ is
protected, and we regard these functions of $T$
as random variables depending on a random tree $T$. 
Thus \eqref{ppt} can now be written 
\begin{equation}
\pp(T)=\P(\Tx\in\cep\mid T).  
\end{equation}
In the annealed version
we more simply consider the probability that a random node in a random tree
$T$ is protected, which equals the expectation 
\begin{equation}
\E \pp(T)=\P(\Tx\in\cep).  
\end{equation}

The first class of random trees that we consider in this paper are the
\sgrt{s}; these are defined using a \ws{} $\wwk$ which we regard as fixed,
see \refS{Strees} for the definition and the connection to conditioned \GWt{s}.
It is well-known that suitable choices of $\wwk$ yield several important
classes of random trees, 
see \eg{} \citet{AldousII}, \citet{Devroye}, \citet{Drmota}, \citet{SJ264}
and \refS{Sex}.

Let 
\begin{equation}
  \Phi(t):=\sumk w_kt^k
\end{equation}
be the generating function of the \ws, and let $\rho\in[0,\infty]$ be its
radius of convergence. We define an important parameter $\tau\ge0$ by:
\begin{romenumerate}
\item $\tau$ is the unique number in $[0,\rho]$ such that
  \begin{equation}\label{tau}
\tau\Phi'(\tau)=\Phi(\tau)<\infty,	
  \end{equation}
if there exists any such $\tau$.
\item 
If \eqref{tau} has no solution, then $\tau:=\rho$. 
\end{romenumerate}
See further \cite[Section 7]{SJ264}, where 
several properties and equivalent
characterizations are given. 
(For example, 
$\tau$ is 
the minimum point in $[0,\rho]$
of $\Phi(t)/t$. 
Furthermore, $\Phi(\tau)<\infty$ also in case (ii), and $\tau>0\iff\rho>0$.)

We define another \ws{} \ppik{} by
\begin{equation}\label{ppik}
  \pi_k:=\frac{w_k\tau^k}{\Phi(\tau)};
\end{equation}
this \ws{} has the generating function
\begin{equation}\label{phit}
  \Phit(t):=\xfrac{\Phi(\tau t)}{\Phi(\tau)}.
\end{equation}
Note that
 $\sumk \pi_k=1$; thus $\ppik$ is a probability distribution on 
the non-negative integers 
$\bbNo:=\set{0,1,2,\dots}$.

\begin{theorem}\label{T1}
Let $\ctn$ be a \sgrt{} with $n$ nodes. Then, with notations as above, 
the following holds as \ntoo.
  \begin{romenumerate}[-10pt]
  \item (Annealed version.)
The probability $p_n=\E \pp(\ctn)$ that a random node in a random tree
$\ctn$ is protected tends to a limit $\px$ as \ntoo, with
\begin{equation}\label{t1}
  \px:=\Phit(1-\pi_0)-\pi_0
=\frac{\Phi\bigpar{\tau-\tau w_0/\Phi(\tau)}-w_0}{\Phi(\tau)}.
\end{equation}

  \item ({Quenched version}.)
The proportion of nodes in $\ctn$ that are protected, 
\ie{} $\pp(\ctn)=\np(\ctn)/n$, converges in probability to $\px$ as \ntoo. 
  \end{romenumerate}
\end{theorem}

The main idea of this paper, viz.\ to study protected nodes by studying
fringe subtrees, applies also to other types of random trees.
We consider binary search trees in \refS{SBST} and random recusive trees in
\refS{SRRT}.

Protected nodes have been studied also for
digital search trees \cite{DuProdinger} and
tries \cite{Gaither+}. 
As far as we know, the fringe subtrees of these random trees have not been
studied in general; this will be dealt with elsewhere.

\begin{ack}
  This research was mainly done during the 
23rd International Meeting on Probabilistic, Combinatorial and Asymptotic
Methods for the Analysis of Algorithms (AofA 2012) in Montreal, June 2012.
We thank the organizers for providing this opportunity and several
participants for helpful comments.
\end{ack}

\section{Simply generated trees and \GWt{s}}\label{Strees}

All trees in this paper are \emph{rooted} and \emph{ordered} (= \emph{plane}).
(For unordered trees, see \refE{Ecayley}.) We denote the outdegree of a
node $v\in T$ by $\dx(v)$. 
Note that a tree is uniquely determined by its
sequence of outdegrees, taken in \eg{} breadth-first order. 
See further \eg{} \cite{Drmota} and \cite{SJ264}.
We let $\st$ denote the set of all ordered rooted trees,
and $\stn:=\set{T\in\st:|T|=n}$ the set of all 
ordered rooted trees with  with $n$ nodes.
By a \emph{random tree} we mean a random element of $\st$ with some given but
arbitrary distribution. (No uniformity is implied unless we say so.)

Given a \ws{} $\wwk$, we define the weight of a tree $T$ to be 
$w(T):=\prod_{v\in T} w_{\dx(v)}$.  
For $n\ge1$, we define the \sgrt{} $\ctn$ as the random
tree obtained by selecting an ordered rooted tree in $\stn$ with probability
proportional to its weight. (We consider only $n$ such that there is at
least one tree in $\stn$ with positive weight.)

It is well-known that \sgrt{s} are essentially the same as \cGWt{s}.
Given  a probability distribution $\ppik$ on $\bbNo$, let $\cT$ be
the corresponding \emph{\GWt}; 
this is a random tree where each node has a random
number of children, and these numbers all are independent and with the
distribution $\ppik$. Furthermore, let $\ctn$ be $\cT$ conditioned on having
exactly $n$ nodes; this is called a \emph{\cGWt}.
(We consider only $n$ such that $\P(|\cT|=n)>0$.)
It is easy to see that the \cGWt{} $\ctn$ coincides with the \sgrt{} defined
using the \ws{} \ppik.
Moreover, if $\wwk$ is any \ws{} with radius of convergence $\rho>0$ (this
is satisfied in virtually all applications), let \ppik{} be given by
\eqref{ppik}. Then the \sgrt{} defined by $\wwk$ coincides with the \sgrt{}
defined by $\ppik$, and thus with the \cGWt{} defined by $\ppik$,
see \eg{} \cite{Kennedy} and \cite{SJ264}.
(There are also other probability distributions yielding the
same \cGWt, but the choice in \eqref{ppik} is the canonical one, 
see \cite{SJ264}.) 

It is easy to see that the probability distribution $\ppik$
has expectation
$\tau\Phi'(\tau)/\Phi(\tau)$, which equals 1 in case (i) above
(\ie, when \eqref{tau} holds),
but is less
than 1 in case (ii) (\ie, when \eqref{tau} has no solution).
Thus,
$\ppik$ yields a
critical \GWt{} $\cT$ in case (i), but 
$\cT$ is subcritical in case (ii). In both cases, $\cT$ is \as{} finite.

\section{Proof of \refT{T1}}\label{Spf}

The proof is based on the fact that the random fringe subtrees of a
\cGWt{} converge in distribution to the corresponding 
(unconditional) \GWt, as stated in the
following theorem. Part (i) was proved by \citet{AldousFringe} 
under some extra conditions,
and by \citet{BenniesK} under fewer extra conditions; 
the general case and (ii)
are proved in \cite[Theorem 7.12]{SJ264}.

\begin{theorem}\label{T2}
Let $\ctn$ be a \sgrt{} with $n$ nodes. Then, with notations as above, 
the following holds as \ntoo.
  \begin{romenumerate} [-10pt]
  \item ({Annealed version}.)
The fringe subtree $\ctnx$ converges in distribution to the \GWt{} $\cT$.
I.e., for every fixed tree $T$, 
\begin{equation}\label{t2a}
  \P(\ctnx=T) \to \P(\cT=T).
\end{equation}
  \item ({Quenched version}.)
The conditional distributions $\cL(\ctnx\mid\ctn)$ converge to the
distribution of $\cT$ in probability.
I.e., for every fixed tree $T$, 
\begin{equation}\label{t2q}
  \P(\ctnx=T\mid\ctn) \pto \P(\cT=T).
\end{equation}
\end{romenumerate} 
\qed
\end{theorem}
Note that the set of (finite) ordered trees is a
countable discrete set; this justifies that it is enough to consider point
probabilities in \eqref{t2a} and \eqref{t2q}.

\begin{proof}[Proof of \refT{T1}]
For the annealed version, it  follows immediately from \eqref{ppt}
and \eqref{t2a}, which can be written
$\ctnx\dto\cT$,
that
\begin{equation}
  p_n=\P(\ctnx\in\cep)\to\P(\cT\in\cep).
\end{equation}

For the quenched version, conditioning on $\ctn$,
we similarly obtain by \eqref{t2q}, 
\begin{equation}
  \pp(\ctn)=\P(\ctnx\in\cep\mid \ctn)\pto\P(\cT\in\cep).
\end{equation}

It remains only to calculate $\P(\cT\in\cep)$. This is easy, by conditioning
on the root degree, $k$ say.
If $k=0$, then the root is a leaf and not protected, and if $k>0$, the root
is protected if and only if each of its $k$ children has at least one child,
which has probability $(1-\pi_0)^k$. Hence, 
\begin{equation}
  \P(\cT\in\cep)
=
\sumki \pi_k(1-\pi_0)^k = \Phit(1-\pi_0)-\pi_0.
\end{equation}
Finally, $\pi_0=w_0/\Phi(\tau)$ by \eqref{ppik}, and
$\Phit(1-\pi_0)=\Phi(\tau-\tau\pi_0)/\Phi(\tau)$ by \eqref{phit}.
\end{proof}

\section{Examples}\label{Sex}

We give several examples of random trees where \refT{T1} applies.
We focus on the calculation of $\px$, since the other conclusions are the
same for all random trees considered here. We omit some steps in the
calculations, see \eg{} \cite[Section 10]{SJ264} for further details.

\begin{example}[ordered trees] \label{EO}
The \ws{} $w_k=1$ yields uniformly random \emph{ordered  trees}. 
In this case, $\Phi(t)=\sumk t^k=1/(1-t)$ and \eqref{tau} has the
  solution $\tau=1/2$, yielding $\pi_k=2^{-k-1}$ (a geometric $\Ge(1/2)$
 distribution) and $\Phit(t)=1/(2-t)$.
Thus $\pi_0=1/2$ and, by \eqref{t1},
\begin{equation}
  \px=\Phit\Bigpar{\frac12}-\frac12=\frac{1}{2-\frac12}-\frac12=\frac16.
\end{equation}
We thus recover from the annealed version in \refT{T1}
the result by \citet{CheonShapiro} that the average
proportion of protected nodes in a random ordered tree converges to 1/6 as
the size goes to infinity.
Moreover, the quenched version shows that holds also for most individual trees.
More precisely, $\pp(\ctn)\pto1/6$, \ie, for any $\eps>0$, the probability
that a uniformly random ordered tree with $n$ nodes has between $(1/6-\eps)n$
and $(1/6+\eps)n$ protected nodes tends to 1 as \ntoo.
\end{example}

\begin{example}[unordered trees]  \label{Ecayley} \label{EU}
We have assumed that the trees are ordered, but we can treat also unordered
labelled trees by giving the children of each node a (uniform) random
ordering.
As is well known, a uniformly random 
\emph{unordered labelled rooted tree} (sometimes called
\emph{Cayley tree}) then becomes simply generated with weights $w_k=1/k!$.
In this case, $\Phi(t)=\sumk t^k/k!=e^t$ and \eqref{tau} has the
  solution $\tau=1$, yielding $\pi_k=e\qw/k!$ (a Poisson $\Po(1)$
 distribution) and $\Phit(t)=e^{t-1}$.
Thus $\pi_0=e\qw$ and, by \eqref{t1},
\begin{equation}
  \px=\Phit\bigpar{1-e\qw}-e\qw=e^{-e\qw}-e\qw\approx0.32432.
\end{equation}
\end{example}

\begin{example}[full $d$-ary trees] 
Uniformly random \emph{full $d$-ary trees}
are \sgrt{s} with $w_k=1$ if $k=0$
or $k=d$, and $w_k=0$ otherwise. 
(Here $d\ge2$ is a fixed integer.
In this case, the number of nodes $n$ has to be $1\pmod d$.)
We have 
$\Phi(t)=1+t^d$ and  $\tau=(d-1)^{-1/d}$, yielding $\pi_0=(d-1)/d$, $\pi_d=1/d$,
and $\Phit(t)=(d-1+t^d)/d$. Consequently, \eqref{t1} yields
\begin{equation}
  \px=\pi_d(1-\pi_0)^d=1/d^{d+1}.
\end{equation}
Thus, \refT{T1} shows that the proportion of 
protected nodes tends to $1/d^{d+1}$.

This was found by \citet{Mansour} (for the annealed version);
note that \cite{Mansour} states  the result  in terms of the
number of internal nodes. Since a full $d$-ary tree with $m$ internal
nodes has $dm+1$ nodes, the proportion of internal nodes that are protected
tends to $1/d^{d}$.

The special case $d=2$ yields full binary trees, for which we find
$\px=1/8$. 
(The proportion $1/4$ given in \cite{CheonShapiro} seems to be a mistake.) 

The special case $d=3$ yields full ternary trees, for which we find
$\px=1/81$, in accordance with \cite{CheonShapiro}.
\end{example}

\begin{example}[$d$-ary trees]
Uniformly random  \emph{$d$-ary trees}
are \sgrt{s} with $w_k=\binom dk$.
(Again, $d\ge2$ is a fixed integer.)
In this case,  
$\Phi(t)=(1+t)^d$ and  $\tau=1/(d-1)$, yielding 
$\pi_k=\binom dk(\frac1d)^k(\frac{d-1}d)^{d-k}$
(a binomial $\Bi(d,1/d)$ distribution) and $\Phit(t)=\xpar{(d-1+t)/d}^d$.
Consequently, $\pi_0=(1-1/d)^d$ and
\begin{equation}
\px
=\parfrac{d-\pi_0}{d}^d-\pi_0^d
=\lrpar{1-\frac{(d-1)^d}{d^{d+1}}}^d-\frac{(d-1)^d}{d^d}.
\end{equation}
In particular, for $d=2$ (binary trees), we obtain $\px=33/64$.
(The proportion $9/256$ given in \cite{CheonShapiro} seems to be a mistake.) 
\end{example}

\begin{example}[Motzkin trees]
A \emph{Motzkin tree} has each outdegree 0,1 or 2. Taking $w_0=w_1=w_2=1$
and $w_k=0$ for $k\ge3$ yields a uniformly random Motzkin tree. We have
$\Phi(t)=1+t+t^2$ and \eqref{tau} has the soultion $\tau=1$, yielding
$\pi_k=1/3$, $0\le k\le2$, and $\Phit(t)=(1+t+t^2)/3$.
Thus, by \eqref{t1}, 
\begin{equation}
  \px=\frac13\lrpar{\frac23+\Bigparfrac23^2}=\frac{10}{27}.
\end{equation}
Hence, the proportion of protected nodes in a uniformly random Motzkin tree
tends to 10/27, as shown (in the annealed version) by \citet{CheonShapiro}.
\end{example}

\section{$\ell$-protected nodes}\label{Sell}

More generally, given an integer $\ell\ge1$, we say that a node in a rooted
tree is \emph{\lprot} if it has distance at least $\ell$ to every leaf that is a
descendant of it. Thus 2-protected = protected
and 1-protected = non-leaf (internal node).

The results above generalize immediately to \lprot{} nodes for
any fixed $\ell\ge1$.
Given a tree $T$, let $\ppl(T)$ be the proportion of nodes in $T$ that are
\lprot, and let $\pxl$ be the probability that the root of the \GWt{} $\cT$
is \lprot.
\begin{theorem}\label{Tl}
Let $\ctn$ be a \sgrt{} with $n$ nodes. Then, with notations as above, 
the following holds as \ntoo.
  \begin{romenumerate}[-10pt]
  \item (Annealed version.)
The probability $p_{n,\ell}=\E \ppl(\ctn)$ that a random node in a random tree
$\ctn$ is $\ell$-protected tends to  $\pxl$ as \ntoo, with
$\pxl$ given by the recursion
\begin{equation}\label{pl}
  \pxl:=\Phit(\pxli)-\pi_0,
\qquad \ell\ge1,
\end{equation}
with $\pxx0=1$ and $\pxx1=1-\pi_0$.

  \item ({Quenched version}.)
The proportion of nodes in $\ctn$ that are $\ell$-protected, 
\ie{} $\ppl(\ctn)$, converges in probability to $\pxl$ as \ntoo. 
  \end{romenumerate}
\end{theorem}
\begin{proof}
  The convergence to $\pxl$ follows from \refT{T2} as in the proof of
  \refT{T1}. The recursion \eqref{pl} follows since the root is \lprot{} if
  and only if it has outdegree $>0$ and each child is $(\ell-1)$-protected.
\end{proof}

\begin{example}
  For uniformly random ordered trees, $\Phit(t)=1/(2-t)$, see \refE{EO},
and thus the recursion \eqref{pl} is
\begin{equation}
  \pxl=\frac1{2-\pxli}-\frac12
=\frac{\pxli}{4-2\pxli},
\qquad \ell\ge1,
\end{equation}
yielding $1/\pxl=4/\pxli-2$ with the solution $1/\pxl=(4^\ell+2)/3$, \ie{}
\begin{equation}
  \pxl=\frac{3}{4^\ell+2},
\qquad \ell\ge0.
\end{equation}
In particular, $\pxx1=1/2$, $\pxx2=1/6$, $\pxx3=1/22$, $\pxx4=1/86$.

Hence, for each fixed $\ell\ge1$, the proportion of \lprot{} nodes in a
uniform random ordered tree tends to $3/(4^\ell+2)$.
\end{example}

\begin{example}
  For uniformly random unordered labelled rooted trees we have by \refE{EU}
$\pi_0=e\qw$ and   $\Phit(t)=e^{t-1}$. Thus \eqref{pl} yields
\begin{align}
\pxx1&=1-e\qw \approx 0.63212,\\
\pxx2&=e^{-e\qw}-e\qw\approx0.32432, \\
\intertext{as in \refE{EU}, and}
  \pxx3&=\exp\lrpar{e^{-e\qw}-e\qw-1}-e\qw \approx0.14093.
\end{align}
\end{example}


\section{Binary search trees}\label{SBST}

A random binary search tree with $n$
nodes is a binary tree obtained by inserting, in the standard manner, 
$n$ independently and identically distributed (\iid) uniform
$[0, 1]$  random variables $X_1,\dots,X_n$
into an initially empty tree, see \eg{} \cite{Drmota}.
Let $\ctn$ be a random binary search tree with $n$ nodes. 
\citet{AldousFringe} 
showed that there exists a random limiting  fringe tree  $\fT$ in this case
too such that \eqref{t2a} and \eqref{t2q} hold (with $\cT$ replaced by
$\fT$); in fact, the convergence in \eqref{t2q} holds a.s.
The limit tree $\fT$ can be described as a binary search tree $\cT_N$ with a
random size $N$; this is easily seen by the recursive construction of the
binary search tree, letting $N$ be the limiting distribution of the subtree
size $|\ctnx|$ of a random node, and a calculation shows that
$\P(N=n)=2/(n+1)(n+2)$, $n\ge1$ \cite{AldousFringe}.
See also \citet{Devroye02} for a simple direct proof.

Moreover, \citet{AldousFringe} also shows that $\fT$ may be constructed as
follows:
Let $\ctt_t$, $t\ge0$, be a random process of a binary tree
growing in continuous time, starting with $\ctt_0$ being a single root, and
adding left and right children with intensity 1 at all possible places. In
other words, given any $\ctt_t$ at a time $t\ge0$, any possible child of an
existing node (excluding children already existing) is added after an
exponential $\Exp(1)$ waiting time; all waiting times being independent.
It is well-known and easy to see that at any fixed time $t>0$, the
conditional distribution of $\ctt_t$ given $|\ctt_t|=n$ equals the
distribution of $\ctn$. Moreover, if we instead take $\ctt_X$ at a random
time $X\sim\Exp(1)$ (independent of everything else), then $\ctt_X\eqd\fT$.

We can now repeat the proof of \refT{T1} and obtain the same results as
above, with
\begin{equation}\label{pbin}
  \begin{split}
\px=\P(\text{the root of $\fT$ is protected})	
=\P(\ctt_X\in\cep)
=\intoo\P(\ctt_t\in\cep)e^{-t}\dd t
  \end{split}
\end{equation}
In order to evaluate $\px$, 
we consider first $\ctt_t$ for a given $t$. The probability, $q_1(t)$ say, 
that the root of
$\ctt_t$ is a leaf is $e^{-2t}$. Similarly, 
if the left child appears at time $s$, then
the probability that it still is a leaf at some later time $t>s$ is
$e^{-2(t-s)}$.  Hence, the probability, $r_1(t)$ say, 
that there is a left child that is a
leaf is 
\begin{equation}
  r_1(t):=
\int_0^t e^{-2(t-s)}e^{-s}\dd s
=\int_0^t e^{s-2t}\dd s = e^{-t}-e^{-2t}.
\end{equation}
The probability that the root has at least one child that is a leaf is thus,
by symmetry and independence,
$1-(1-r_1(t))^2= 2r_1(t)-r_1(t)^2 $ and the probability that the root in
$\ctt_t$ is not protected is
\begin{equation}
  \begin{split}
q_1(t)+ 2r_1(t)-r_1(t)^2
&=
e^{-2t}+2e^{-t}-2e^{-2t} -(e^{-t}-e^{-2t} )^2
\\&
=2e^{-t}-2e^{-2t} +2e^{-3t}-e^{-4t}.
  \end{split}
\end{equation}
Hence we obtain from \eqref{pbin}
\begin{equation}
  \px=1-\intoo \bigpar{2e^{-t}-2e^{-2t} +2e^{-3t}-e^{-4t}}e^{-t}\dd t
=\frac{11}{30},
\end{equation}
in accordance with \citet{MahmoudWard-bst} and \citet{Bona}.

More generally, let $q_\ell(t)$ be the probability that the root of $\ctt_t$ 
is \emph{not} \lprot, $\ell\ge1$, and let $r_\ell(t)$ be the probability
that the root in $\ctt_t$ has a left child that is not \lprot.
The same argument as above yields the recursion, for $\ell\ge2$,
\begin{align}
  q_\ell(t)&=q_1(t)+2r_{\ell-1}(t)-r_{\ell-1}(t)^2,\\
  r_{\ell-1}(t)&=\intot q_{\ell-1}(t-s)e^{-s}\dd s
=e^{-t}\intot q_{\ell-1}(s)e^{s}\dd s,
\end{align}
and then the asymptotic proportion of \lprot{} nodes is found as
\begin{equation}
  \pxl=1-\intoo q_\ell(t)e^{-t}\dd t.
\end{equation}
A Maple calculation yields $\pxx1=2/3$, $\pxx2=11/30$,
$\pxx3={\xfrac {1249}{8100}}$,   
$\pxx4={\xfrac {103365591157608217}{2294809143026400000}}
\approx 0.04504$, 
in agreement with \citet{Bona}, who calculates these values by a different
method. 

\begin{remark}
\citet{Bona} considers $c_\ell$, the asymptotic probability that a
random node is at \emph{level} $\ell$, meaning that the distance to the
nearest leaf that is a descendant
is $\ell-1$; thus a node is $\ell$-protected if it is at a
level strictly larger than $\ell$, and $c_\ell=\pxx{\ell-1}-\pxx\ell$, with
$\pxx0=1$.  
\end{remark}

In the quenched version,
asymptotic normality of the number of protected nodes
was shown by  \citet{MahmoudWard-bst}. 
Alternatively, this follows easily by the method of \citet{Devroye02},
see \cite{HolmgrenJanson} for details.

\section{Random recursive trees} \label{SRRT}
A uniform random recursive tree (URRT) $\ctn$ of order $n$
 is a tree with $n$ nodes labeled \set{1, \dots , n}. The root is labelled
1, and for $2 \le i \le n$, the node labelled $i$ chooses a vertex in 
\set{1, \dots, i-1} uniformly at random as its parent.
See \eg{} \cite{DevroyeFF}, \cite{Drmota}, \cite{SmytheMahmoud}. 
   This case is very similar to the random binary search tree in \refS{SBST}:
\citet{AldousFringe} has shown the existence of a random limiting fringe tree
$\fT$, and again $\fT$ can be described as $\cT_N$, now with
$\P(N=n)=1/n(n+1)$.
Moreover, $\fT$ can be constructed as $\ctt_X$ with $X\sim\Exp(1)$ in this
case too, where now $\ctt_t$ is the random tree process where each node gets
a new child with i.i.d.\ exponential waiting times with intensity 1.
(The Yule tree process.)

The children of the root arrive in a Poisson process with intensity 1; hence
the number of children of the root in $\ctt_t$ has the distribution
$\Po(t)$,
and the probability that the root is a leaf is 
$\P(\Po(t)=0)=e^{-t}$.
Moreover, a child that is born at time $s$ is still a leaf at time $t>s$
with probability $e^{-(t-s)}$. Hence children of the root that remain leaves
at time $t$ are born with intensity $e^{-(t-s)}$, $s\in(0,t)$, and since a
  thinning of a Poisson process is a Poisson process, it follows that the
  number of children of the root that are leaves at time $t$ has a Poisson
  distribution with expectation
$\intot e^{-(t-s)}\dd s=1-e^{-t}$. Consequently, the probability that the
  root of $\ctt_t$ has no child that is a leaf is
$\exp(-(1-e^{-t}))$. Subtracting the probability that the root has no child
  at all, we obtain the probability $p_2(t)$ that the root of $\ctt_t$ is
  protected as
  \begin{equation}
	p_2(t)=\exp\bigpar{e^{-t}-1}-e^{-t}
  \end{equation}
and thus
  \begin{equation}
	\begin{split}
\px&=\intoo	p_2(t) e^{-t}\dd t
=\intoo \exp\bigpar{e^{-t}-1}e^{-t}\dd t- \intoo e^{-2t}\dd t
\\&
=\intoi \exp(x-1)\dd x-\frac12
=
\frac12-e^{-1},	  
	\end{split}
  \end{equation}
in accordance with \citet{MahmoudWard-rrt}.

We can treat \lprot{} nodes too in random recursive trees by the same
method. If $p_\ell(t)$ is the probability that the root is \lprot{} in
$\ctt_t$, and $q_\ell(t)=1-p_\ell(t)$, then the number of children of the
root that are not $(\ell-1)$-protected at time $t$ is Poisson distributed
with mean $\intot q_{\ell-1}(t-s)\dd s=\intot q_{\ell-1}(s)\dd s$, yielding
the recursion, for $\ell\ge1$,
\begin{equation}
  \begin{split}
p_\ell(t)=\exp\lrpar{-\intot q_{\ell-1}(s)\dd s}-\exp(-t)	
=
e^{-t}\lrpar{ \exp\lrpar{\intot p_{\ell-1}(s)\dd s}-1},
  \end{split}
\end{equation}
with $p_0(t)=1$ and $p_1(t)=1-e^{-t}$. 
In principle, $\pxl$ can be computed as $\intoo p_\ell(t)e^{-t}\dd t$,
but in this case we do not know any closed form for $\ell>2$.



\newcommand\AAP{\emph{Adv. Appl. Probab.} }
\newcommand\JAP{\emph{J. Appl. Probab.} }
\newcommand\JAMS{\emph{J. \AMS} }
\newcommand\MAMS{\emph{Memoirs \AMS} }
\newcommand\PAMS{\emph{Proc. \AMS} }
\newcommand\TAMS{\emph{Trans. \AMS} }
\newcommand\AnnMS{\emph{Ann. Math. Statist.} }
\newcommand\AnnPr{\emph{Ann. Probab.} }
\newcommand\CPC{\emph{Combin. Probab. Comput.} }
\newcommand\JMAA{\emph{J. Math. Anal. Appl.} }
\newcommand\RSA{\emph{Random Struct. Alg.} }
\newcommand\ZW{\emph{Z. Wahrsch. Verw. Gebiete} }
\newcommand\DMTCS{\jour{Discr. Math. Theor. Comput. Sci.} }

\newcommand\AMS{Amer. Math. Soc.}
\newcommand\Springer{Springer-Verlag}
\newcommand\Wiley{Wiley}

\newcommand\vol{\textbf}
\newcommand\jour{\emph}
\newcommand\book{\emph}
\newcommand\inbook{\emph}
\def\no#1#2,{\unskip#2, no. #1,} 
\newcommand\toappear{\unskip, to appear}

\newcommand\urlsvante{\url{http://www.math.uu.se/~svante/papers/}}
\newcommand\arxiv[1]{\url{arXiv:#1.}}
\newcommand\arXiv{\arxiv}

\def\nobibitem#1\par{}


\begin{thebibliography}{99}

\bibitem[Aldous(1991)]{AldousFringe} 
David Aldous,
Asymptotic fringe distributions for general families of random trees.
\emph{Ann. Appl. Probab.} \vol1 (1991), no. 2, 228--266.

\bibitem[Aldous(1991)]{AldousII} 
David Aldous, 
The continuum random tree II: an overview.
\emph{Stochastic Analysis (Durham, 1990)}, 23--70, 
London Math. Soc. Lecture Note Ser. 167, Cambridge Univ. Press, 
Cambridge, 1991. 

\bibitem[Bennies and Kersting(2000)]{BenniesK}
J\"urgen Bennies and G\"otz Kersting, 
A random walk approach to Galton--Watson trees. 
\emph{J. Theoret. Probab.} \vol{13} (2000), no. 3, 777--803. 

\bibitem[B{\'o}na(2013)]{Bona}
Mikl\'os B{\'o}na,
$k$-protected vertices in binary search trees.
\emph{Adv. Appl. Math.}, to appear.
\arXiv{1304.6105}


\bibitem[Cheon and  Shapiro(2008)]{CheonShapiro}
Gi-Sang Cheon and  Louis W.  Shapiro,
Protected points in ordered trees. 
\emph{Applied Mathematics Letters}
\vol{21} (2008), no. 5, 516--520.

\bibitem[Devroye(1998)]{Devroye}
Luc Devroye,
Branching processes and their applications in the analysis of tree
structures and tree algorithms.
\inbook{Probabilistic Methods for Algorithmic Discrete Mathematics},
eds. M. Habib, C. McDiarmid, J. Ramirez and B. Reed, 
Springer, Berlin, 1998,
pp. 249--314.

\bibitem[Devroye(2002)]{Devroye02}
Luc Devroye, 
Limit laws for sums of functions of subtrees of random binary search trees. 
\emph{SIAM J. Comput.} \vol{32} (2002/03), no. 1, 152--171.

\bibitem{DevroyeFF}
Luc Devroye, Omar Fawzi and Nicolas Fraiman,
Depth properties of scaled attachment random recursive trees. 
\emph{Random Structures Algorithms} \textbf{41} (2012), no. 1, 66--98.

\bibitem[Drmota(2009)]{Drmota}
Michael Drmota,
\book{Random Trees},
Springer, Vienna, 2009.

\bibitem[Du and Prodinger(2012)]{DuProdinger}
Rosena R.X. Du and Helmut Prodinger,
Notes on protected nodes in digital search trees.
\emph{Applied Mathematics Letters}
\vol{25} \no6 (2012),  1025--1028.

\bibitem{Gaither+}
Jeffrey Gaither, Yushi Homma, Mark Sellke and Mark Daniel Ward, 
On the number of 2-protected nodes in tries and suffix trees.
\emph{Proceedings, 2012 Conference on Analysis of Algorithms, AofA '12
(Montreal, 2012)},
\emph{DMTCS Proceedings} 2012, 
381--398.

\bibitem[Holmgren and Janson(2013+)]{HolmgrenJanson}
Cecilia Holmgren and Svante Janson,
Limit laws for functions of fringe trees for binary search trees and
recursive trees. 
In preparation.

\bibitem[Janson(2012)]{SJ264}
Svante Janson,
Simply generated trees, conditioned Galton--Watson trees,
 random allocations and condensation.
\emph{Probability Surveys} \vol9 (2012), 103--252.

\bibitem[Kennedy(1975)]{Kennedy}
Douglas P. Kennedy,
The Galton--Watson process conditioned on the total progeny.
\JAP \vol{12}  (1975),  800--806. 


\bibitem[Mahmoud and Ward(2012+)]{MahmoudWard-bst}
Hosam H. Mahmoud and Mark Daniel Ward,
Asymptotic distribution of two-protected nodes in random binary search
trees.
\emph{Applied Mathematics Letters}
\vol{25} (2012), no. 12,  2218--2222.

\bibitem[Mahmoud and Ward(2013)]{MahmoudWard-rrt}
Hosam H. Mahmoud and Mark Daniel Ward,
Asymptotic properties of protected nodes in random recursive trees.
Preprint, 2013. 

\bibitem[Mansour(2011)]{Mansour}
 Toufik Mansour, 
Protected points in $k$-ary trees. 
\emph{Applied Mathematics Letters}
\vol{24} (2011), no. 4, 478--480.

\bibitem{SmytheMahmoud}
Robert T. Smythe and  Hosam M.  Mahmoud,
A survey of recursive trees. 
\emph{Theory Probab. Math. Statist.} \textbf{51} (1995), 1--27 (1996);
translated from 
\emph{Teor. \u Imov\=\i r. Mat. Stat.} \textbf{51} (1994), 1--29 (Ukrainian).


\end{thebibliography}
\end{document}